\documentclass[11pt]{article}

\usepackage[a4paper,top=1.4cm,bottom=2cm,left=2.5cm,right=2.5cm]{geometry}

\usepackage[dvipsnames, table]{xcolor}

\usepackage[pdftex]{graphicx}

\usepackage{epstopdf}

\usepackage[cmex10]{amsmath}
\usepackage{amsfonts, amssymb, amsthm, mathtools}
\usepackage{empheq}

\usepackage{multirow}   		
\usepackage{booktabs}   		
\usepackage{array}
\newcolumntype{+}{>{\global\let\currentrowstyle\relax}}
\newcolumntype{^}{>{\currentrowstyle}}

\usepackage[sort,compress]{cite}
\usepackage{microtype}
\usepackage{algorithm}

\usepackage[numbers]{natbib}
\usepackage{standalone}
\usepackage{graphicx}
\usepackage{subcaption}
\usepackage{graphicx} 
\usepackage{amsmath, amssymb, amsthm}
\usepackage{algorithm}
\usepackage{algpseudocode}

\usepackage{xcolor}
\DeclareMathOperator{\argmin}{argmin}

\DeclareMathOperator{\prox}{prox}

\newtheorem{theorem}{Theorem}
\newtheorem{assumptions}{Assumption}
\newtheorem{assumption}[assumptions]{Assumption}
\newtheorem{lemma}{Lemma}
\newtheorem{corollary}{Corollary}

\def\xinit{x_0}

\def\x{{x}}
\def\F{{f}}
\def\G{{g}}

\def\d{{n}}
\def\dtheta{{d}}
\def\ddata{{m}}
\def\acronym{\texttt{ReTune}}
\def\gdot{{\tiny\bullet}}
\def\J{\mathrm{J}}
\def\JF{\mathrm{JF}}

\def\Id{\mathrm{Id}}

\def\NN{\mathbb{N}}

\title{Restarted contractive operators to learn at equilibrium}

\author{Leo Davy\thanks{L. Davy (\texttt{leo.davy@ens-lyon.fr}) and N. Pustelnik (\texttt{nelly.pustelnik@ens-lyon.fr}) are with Laboratoire de Physique de l'ENS Lyon, CNRS UMR 5672, F69007 Lyon, France. L. Davy is supported by DGA/AID (01D22020572) and IADoc@UdL (ANR-20-THIA-0007-01). N. Pustelnik  is supported by Fondation Simone et Cino Del Duca - Institut de France.}\and Luis M. Brice\~no-Arias\thanks{L. Brice\~no-Arias (\texttt{luis.briceno@usm.cl}) is with Universidad Técnica Federico Santa Maria,
    Departamento de Matem\'atica,
    8940897, San Joaquín, Santiago de Chile } \and Nelly Pustelnik$^*$}

\begin{document}

\maketitle

\begin{abstract}
Bilevel optimization offers a methodology to learn hyperparameters in imaging inverse problems, yet its integration with automatic differentiation techniques remains challenging. On the one hand, inverse problems are typically solved by iterating arbitrarily many times some elementary scheme which maps any point to the minimizer of an energy functional, known as equilibrium point. On the other hand, introducing parameters to be learned in the energy functional yield architectures very reminiscent of Neural Networks (NN) known as Unrolled NN and thus suggests the use of Automatic Differentiation (AD) techniques. Yet, applying AD requires for the NN to be of relatively small depth, thus making necessary to truncate an unrolled scheme to a finite number of iterations. First, we show that, at the minimizer, the optimal gradient descent  step computed in the Deep Equilibrium (DEQ) framework admits an approximation, known as Jacobian Free Backpropagation (JFB), that is much easier to compute and can be made arbitrarily good by controlling Lipschitz properties of the truncated unrolled scheme.  Second, we introduce an algorithm that combines a restart strategy with JFB computed by AD and we show that the learned steps can be made arbitrarily close to the optimal DEQ framework. Third, we complement the theoretical analysis by applying the proposed method to a variety of problems in imaging that progressively depart from the theoretical framework. In particular we show that this method is effective for training weights in weighted norms; stepsizes and regularization levels of Plug-and-Play schemes; and a DRUNet denoiser embedded in Forward-Backward iterates.
\end{abstract}


\section{Introduction}
Inverse problems in imaging have seen significant advancements over the past 20 years, evolving from estimation techniques based on the minimization of non-smooth, possibly non-convex problems to model-based neural network procedures (e.g., unrolled or Plug-and-play)  that combine traditional minimization techniques with powerful neural networks.

Despite the variety of methods available to solve inverse problems, a persistent challenge remains the accurate estimation of the underlying hyperparameters. This can be addressed through automatic differentiation when using standard neural networks. However, it demands special attention to combine the benefits of automatic differentiation when employed for parameter estimation of model-based neural networks. This particular challenge will be at the core of our contribution.

\noindent \textbf{Bilevel formulation --} The standard bilevel optimization formulation aims to estimate the \textit{optimal} set of  hyperparameters $\theta^*$ involved in an estimator $\widehat{x}_{\theta}$ leading to the optimal solution $\widehat{x}_{\theta^*}$. Formally, we aim to estimate
\begin{align}
    &\theta^* \in \textrm{Argmin}_{\theta} \mathcal{L}(\widehat{x}_\theta) \tag{Outer Problem}\label{eq:outer_problem} \\
    &\text{with }\quad \widehat{x}_\theta \in  \textrm{Argmin}_x E_\theta(x). \tag{Inner Problem}\label{eq:inner_problem}
\end{align}

In our contribution, we mainly focus on a specific form for the \ref{eq:inner_problem} energy that is
\begin{equation}\label{eq:min_with_theta_F_plus_G}
    E_\theta(x) = \F_\theta(x) + \G_\theta(x),
\end{equation}
where $\F_\theta$ and $\G_\theta$ are functions from $\mathbb{R}^\d$ to $]-\infty,+\infty]$. 
Many inverse problems in imaging can be framed in the aforementioned form, where $\F_\theta$  models a data-fidelity term, $\G_\theta$ a prior \cite{Combettes2011,Chambolle_A_2016_an,Pustelnik_N_20016_j-w-enc-eee_wav_bid}, and $\theta \in \mathbb{R}^{\dtheta}$ models hyperparameters involved in the definition of $\F_\theta$ or $\G_\theta$ such as the regularization parameter, a weight matrix, a linear operator, and so on. Many iterative schemes have been developed in order to minimize \eqref{eq:min_with_theta_F_plus_G} when no closed-form solution is available. Among these schemes, the proximal algorithms will be of particular interest to this contribution (see e.g. \cite{Livre1, Combettes2011, Chambolle_A_2016_an}).

The bilevel formulation aims to estimate  the underlying parameters $\theta$ automatically, rather than relying on manual tuning that becomes intractable when the dimension $\dtheta>1$.

\noindent \textbf{Specification to the inverse problem framework -- }In \eqref{eq:outer_problem}, $\mathcal{L}$ is a loss function used to fit the parameters $\theta$ involved in $\widehat{x}_\theta$. The most encountered framework is certainly supervised learning defined from a dataset $\mathcal{D}=\{(\overline{x}^{(i)}, y^{(i)}) \}_{i\in \mathbb{I}} \subset \mathbb{R}^\d\times\mathbb{R}^\ddata$ following some distribution.  When considering inverse problems in imaging,  $\overline{x}^{(i)}$ represents the ground-truth image and  $y^{(i)}$ the observation. The standard relation between $y^{(i)}$ and $\overline{x}^{(i)}$ is given by $y^{(i)} = \mathrm{A} \overline{x}^{(i)} + \epsilon$ where $\mathrm{A}$ models a linear operator and $\epsilon$ a random additive noise. The goal is then to obtain an estimator $\widehat{x}_\theta(y^{(i)})$ the closest from $\overline{x}^{(i)}$ by minimizing a loss, e.g.,
\begin{equation}\label{eq:mse}
    \mathcal{L}(\widehat{x}_\theta) = \sum_{i\in \mathbb{I}} ||\overline{x}^{(i)} - \widehat{x}_\theta(y^{(i)})||_2^2.
\end{equation}
In the following, we write $\widehat{x}_\theta$ instead of $\widehat{x}_\theta(y^{(i)})$, omitting the explicit dependence on $y^{(i)}$. $\mathcal{L}$ can also model an unsupervised loss. The reader could refer to \cite{ramani_monte-carlo_2008, tachella_unsure_2024} for more details about unsupervised settings. In the following, we assume that $\mathcal{L}$ does not depend directly on $\theta$.\\
\noindent \textbf{Limitations of bilevel formulation -- } The \ref{eq:outer_problem} relies on the estimation of $\widehat{x}_\theta$, which is obtained by a minimization strategy, typically a gradient descent method (or accelerated versions of it) leading to the following iterations:
\begin{equation}\label{eq:gradient_descent_nn}
    \theta^{[\ell + 1]} = \theta^{[\ell]} - \tau \partial_\theta \mathcal{L}(\widehat{x}_{\bullet})(\theta^{[\ell]}),
\end{equation}
where $\partial_{\theta}$ is the limiting subdifferential (see Section~\ref{sec:subdiff}) and $\partial_\theta \mathcal{L}(\widehat{x}_{\bullet})(\theta^{[\ell]})$ is known as the hypergradient.

However, in practice a minimizer of \eqref{eq:inner_problem} is  never computed exactly and is usually approximated by $K$ iterations of an algorithm allowing to minimize \eqref{eq:min_with_theta_F_plus_G}. Formally, 
\begin{equation}\label{eq:def_Phi_K_approx}
    \widehat{x}_\theta \approx  \widetilde{x}_{\theta,K} = \varphi_K (\cdot,\theta) \circ \dots \circ \varphi_1 (\xinit,\theta),
\end{equation}
where $\xinit\in \mathbb{R}^{\d}$ and $\varphi_k$ denotes the building block of the iterative scheme (see Section~\ref{sec:convergent_algorithms}). 

 
 In this work, we develop the theoretical framework needed to solve  the bilevel optimization problem, in a context where only a \emph{truncated unrolled} scheme - as defined in \eqref{eq:def_Phi_K_approx} - is available. Here the number of iterations $K$ is assumed to be fixed and relatively small, making it possible to apply \textit{automatic differentiation} (AD) techniques \cite{griewank_automatic_1989,bolte_nonsmooth_2021}.

\noindent \textbf{State-of-the-art -- } Many lines of research have been considered for bilevel optimization. A generic and popular framework is given by the Deep-Equilibrium (DEQ) formalism \cite{bai_deep_2019, gilton_deep_2021,crockett_bilevel_2022}. This formalism, recalled in Section~\ref{sec:subdiff}, allows to derive an expression of the hypergradient by leveraging a fixed-point problem associated to \eqref{eq:min_with_theta_F_plus_G}. However, its implementation can be challenging and costly, primarily due to the necessity to reach $\widehat{x}_\theta$ and, secondly, to the need of inverting  a high-dimensional Jacobian matrix. A simplified DEQ framework known as Jacobian-Free Backpropagation (JFB) has been proposed in \cite{fung_jfb_2022} and its efficiency on practical problems has made of JFB a popular alternative to DEQ. In the context of inverse problems in imaging, DEQ framework has been smartly employed in \cite{gilton_deep_2021, zou_kamilov_deq_sigproc_2023}.

Another class of methods for bilevel optimization, known as iterative differentiation, alternates optimization steps for the inner problem with gradient-based updates for the outer problem \cite{deledalle_stein_2014, pascal_automated_2021, suonpera_valkonen_linearly_2024, suonpera_valkonen_single-loop_nodate, bertrand_implicit_2020, bertrand_implicit_2022, pouliquen_implicit_2023}. Provided that the iterations of the inner minimization algorithm satisfy certain smoothness conditions, this strategy enables access to a sequence of  $\partial_{\theta} \varphi(\cdot, \theta)$ through the application of the chain rule \cite{bertrand_implicit_2022}. However, its practical implementation, often non-trivial, requires careful hyperparameter tuning, and is strongly problem-dependent.  The effect of inexact solutions in both the inner and outer optimization problems has been further analyzed in \cite{bogensperger_adaptively_2025}.

We also note the works \cite{jiu_deep_2021, le_faster_2022, nguyen_map-informed_2023, savanier_deep_2023}, which focus on directly learning the parameters of the truncated unrolled scheme by AD. However, these approaches generally come with limited theoretical guarantees concerning the output and its connection to an underlying minimization problem.

\noindent \textbf{Contributions and organization of the paper } -- In this brief introduction, we have introduced several important concepts involved in bilevel formalism for inverse problem solving, namely: \textit{truncated unrolled}, Deep-Equilibrium, and Jacobian Free Backpropagation. In the present contribution, we propose a new algorithm named \acronym, that stands for \textit{Restart Truncated Unrolled}, which offers a more straightforward implementation than DEQ to perform an accurate estimation of $\partial_\theta \mathcal{L}(\widehat{x}_{\bullet})$. Our approach ensures that JFB, as introduced in \cite{fung_jfb_2022}, is effectively integrated within the context of \textit{truncated unrolled} models, thereby establishing a connection between DEQ and \textit{truncated unrolled} learning. 

 In Section~\ref{sec:basics_of_optim}, we review several fundamental results and definitions from optimization theory  allowing us to solve \eqref{eq:min_with_theta_F_plus_G} for fixed parameters. Particular attention is given to the notion of Lipschitz continuity, which plays a key role in establishing convergence guarantees for iterative methods. We also introduce the concept of the truncated unrolled algorithm. Finally, both theoretical and practical aspects of bilevel optimization will be discussed. 
 
In Section~\ref{sec:section3}, we introduce the \acronym\ algorithm, which consists of restarting (power-iterates of) truncated unrolled scheme and perform AD over the final restarted step. Under technical assumptions, we prove that \acronym\ converges to a solution of \eqref{eq:min_with_theta_F_plus_G}. We then study the relationship between DEQ, JFB, and \acronym\ by stating and proving bounds that quantify the error in the gradient steps across these different frameworks. 




In Section~\ref{sec:section4}, we numerically assess the efficiency of \acronym\ on two types of experimental settings. The first one focuses on a denoising scenario that satisfies part of the technical assumptions provided in Section~\ref{sec:section3}. The second one departs from this ideal theoretical framework to explore more challenging image restoration problems involving denoisers defined by neural networks. 


\section{Basics of optimization and learning}\label{sec:basics_of_optim}
\subsection{From convergent algorithms to unrolled neural networks}\label{sec:convergent_algorithms}
\textbf{Convergent algorithms -- }We consider minimization problems of the form:
\begin{equation}\label{eq:min_F_plus_G}
    \widehat{x} \in \mathrm{Argmin}_{x\in\mathbb{R}^\d} \F(x) + \G(x),
\end{equation}
where $f\colon \mathbb{R}^\d\to\mathbb{R}$ is convex differentiable with 
$L$-Lipschitz continuous gradient, $g\in \Gamma_0(\mathbb{R}^\d)$,
and we assume existence of solutions.
In general, \eqref{eq:min_F_plus_G} does not admit closed-form solutions and solving it amounts to construct a sequence of approximate solutions $(x_k)_{k\in \mathbb{N}}$ that converges to $\widehat{x}$. For instance, \eqref{eq:min_F_plus_G} can be solved via
the forward-backward splitting (or proximal gradient), which is defined by the Picard iteration for $x_0\in\mathbb{R}^\d$ through
\begin{equation}
\label{e:Picard}
(\forall k\in\NN)\quad x_{k+1}:=\varphi_{k+1}(x_k),    
\end{equation}
where
\begin{align}
 (\forall k\in\NN)\quad   \varphi_k &:= \prox_{\tau_kg} \circ (\mathrm{I}-\tau_k\nabla f),
\end{align}
$\tau_k>0$ is the stepsize,   
$\nabla f$ is the gradient of $f$, and $\prox_g\colon x\mapsto \argmin_{y\in\mathbb{R}^\d}\big(g(y)+\|y-x\|^2/2\big)$ is the proximity operator of $g$. We have
\begin{equation}
\label{e:fixedpoint}
 \mathrm{Argmin}_{x\in\mathbb{R}^\d} \F(x) + \G(x) = \bigcap_{k\in\NN}{\rm fix}\,\varphi_k,    
\end{equation}
where ${\rm fix}\,\varphi:=\{x\in\mathbb{R}^\d\,:\,x=\varphi (x)\}$.
Moreover, under the assumption 
$0<\inf_{k\in\NN}\tau_k\le\sup_{k\in\NN}\tau_k<2/L$, the convergence of the sequence generated by the algorithm to a solution to \eqref{eq:min_F_plus_G} is guaranteed \cite{Livre1}.
If $f$ is $\mu$-strongly convex, \eqref{eq:min_F_plus_G}~has a a unique solution $\widehat{x}$ and the operator $\varphi_k$ turns out to be 
$\omega(\tau_k)$-Lipschitz continuous \cite{TaylorHen}, where
$$
(\forall \tau\in]0,2/L[)\quad \omega(\tau)=\max\{|1-\tau \mu|,|1-\tau L|\}\in\left]0,1\right[.
$$
In this strongly convex setting, the optimal step-size and rate are $\tau_k\equiv\tau^*=2/(\mu+L)$ and $\omega^*=\omega(\tau^*)=(L-\mu)/(L+\mu)$, 
respectively. Under strong convexity assumptions on $g$, the Lipschitz constant can be improved and alternative 
Lipschitz constants are also obtained for various algorithms as Peaceman-Rachford, Douglas-Rachford, and Gradient descent under strong convexity assumptions. The reader is referred to \cite{Sigpro1} for more details. 

In a more general context, depending on the assumptions on the functions $f$ and $g$, several algorithms generate
a sequence $\{x_k\}_{k\in\NN}$ using the \textit{elementary steps} in the recurrence 
\eqref{e:Picard} with a particular choice  $\varphi_k\colon\mathbb{R}^\d\to\mathbb{R}^\d$ satisfying \eqref{e:fixedpoint}. In the following, we denote $\Phi_K\colon\mathbb{R}^\d\to\mathbb{R}^\d$ the composition of $K$ elementary steps  allowing to construct $\widetilde{x}_K$ from an initialization $\xinit$. More precisely,
\begin{equation}\label{eq:def_Phi_K}
 (\forall K\in\NN)\qquad  \widetilde{x}_K = \Phi_K(\xinit) = \varphi_K \circ \dots \circ \varphi_1 (\xinit),
\end{equation}
It is easy to verify from \eqref{e:fixedpoint} that, if, for every $k\in\NN$, $\varphi_k$ is $\omega_k$-Lipschitz continuous with some $\omega_k>0$, then 
\begin{equation}\label{eq:phi_k_leq_delta_k}
 (\forall \xinit \in \mathbb{R}^\d) \qquad ||\Phi_K(\xinit) - \widehat{x}|| \leq \delta_K ||\xinit - \widehat{x}||, 
\end{equation}
where $\delta_K\le\prod_{k=1}^K\omega_k$. If $\delta_K\in\left]0,1\right[$, the sequence generated by the recurrence \eqref{e:Picard} with function $\Phi_K$ converges to $\widehat{x}$ with a linear convergence rate of $\delta_K$ in view of the Banach-Picard theorem~\cite[Theorem 1.50]{Livre1}.


%

\noindent\textbf{Unrolled proximal neural network -- } A very popular approach to create parameterized estimates is to obtain them through a Neural Network (NN). Formally, a NN reads
\begin{equation}\label{eq:neural_network}
    \widetilde{x}_{\theta,K} = \sigma_K(W_K\dots (\sigma_1(W_1 y + b_1))\dots+b_k)
\end{equation}
with non-linear (activation) functions $\sigma_k$, linear maps $W_k$, and bias $b_k$. The underlying hyperparameters are typically the weights and the bias i.e. $\theta= \{W_1,\ldots, W_K,b_1,\ldots b_K\}$. 
Designing a good NN architecture remains challenging since it has a large impact on performance and stability, both at training and at inference.  In standard NN formulation,  $\widetilde{x}_{\theta,K}$  does not result from any fixed point architecture. Unrolled  NN have been proposed as an alternative  to standard NN to incorporate model knowledge. Formally, starting from \eqref{eq:def_Phi_K}, an unrolled NN reads
\begin{equation}
\label{eq:Tune}
    \widetilde{x}_{\theta,K} = \Phi(\xinit,\theta) =  \varphi_K (\cdot,\theta) \circ \dots \circ \varphi_1 (\xinit,\theta).
\end{equation}
where $\varphi_k$ is constructed from fixed point analysis and contains weights matrix, bias, and non-linearities (e.g. LISTA \cite{lecun_lista}) so that an unrolled NN exhibits the same structure as \eqref{eq:neural_network}. This method has been shown to be effective on a variety of problems in image processing but often lack convergence guarantees. 

The numerical necessity of choosing $K$ small means that \eqref{eq:Tune} does not exactly solve the \ref{eq:inner_problem}. In order to emphasize the truncated aspect of the scheme, we will refer to \eqref{eq:Tune} as the  \textit{truncated unrolled} procedure over $K$ steps. Our contribution is to introduce an algorithmic scheme involving \textit{truncated unrolled} iterations that comes with convergence guarantees.

\subsection{From theoretical bilevel to practical bilevel}\label{sec:subdiff}
This section focuses on the second key element in  bilevel optimization, namely the definition of the hypergradient $\partial_\theta \mathcal{L}(\widehat{x}_{\bullet})(\theta^{[\ell]})$ in \eqref{eq:gradient_descent_nn}.

\noindent \textbf{Recall about subdifferentiality -- }For every proper and lower semicontinuous function $f : \mathbb{R}^\d \to (-\infty, +\infty]$, we define:
\begin{itemize}
  \item[(i)] the \emph{Fréchet subdifferential} of $f$ at $x\in \mathrm{dom}\, f$, written $\widehat{\partial} f(x)$, as the set of all vectors $u \in \mathbb{R}^\d$ satisfying
  \[
  \liminf_{\substack{y \neq x \\ y \to x}} \frac{f(y) - f(x) - \langle u, y - x \rangle}{\|y - x\|} \geq 0.
  \]
  When $x \notin \mathrm{dom}\, f$, we set $\widehat{\partial} f(x) = \emptyset$.
  
  \item[(ii)] The \emph{limiting-subdifferential}~\cite{Boris}, or simply the subdifferential, of $f$ at $x \in \mathbb{R}^\d$, written $\partial f(x)$, is defined through the following closure process:
  \[
  \partial f(x) := \left\{ u \in \mathbb{R}^\d : \exists x^k \to x,\ f(x^k) \to f(x)\ \text{and}\ \widehat{\partial} f(x^k) \ni u^k \to u\right\}.
  \]
\end{itemize}
We also denote by $\partial_{x_i} f(x)$ to the limiting subdifferential of the function $\xi\mapsto f(x_1,\ldots, x_{i-1},\xi,x_{i+1},\ldots, x_\d)$. In the case of
vector-valued locally Lipschitz continuous functions $F\colon\mathbb{R}^\d\to \mathbb{R}^N$, the Rademacher's theorem asserts that $F$ is almost everywhere differentiable and we denote by $\Omega_F$ the negligible set of points at which $F$ is not differentiable. The generalized Jacobian is defined by
$$\partial F\colon x\mapsto {\rm co}\left\{\lim_{x_k\to x}{\rm J}F(x_k)\,|\, \{x_k\}_{k\in\NN}\subset\mathbb{R}^\d\setminus\Omega_F\right\},$$
where ${\rm co}$ stands for the convex hull, ${\rm J}F(x)\in\mathbb{R}^{N\times \d}$ is the classical Jacobian matrix in $\x\notin \Omega_F$, and the space of $N\times \d$
real matrices is endowed with the norm $\|\cdot\|_2$.
It follows from \cite[Proposition~2.6.2]{Clark} and\footnote{$\|\cdot\|_F$ stands for the Frobenius norm.} $\|\cdot\|_2\le\|\cdot\|_F$ that 
if $F$ is $\omega$-Lipschitz at $x\in\mathbb{R}^\d$, for some $\omega >0$, we have, for every  
$g\in \partial F(x)$,  $\|g\|_2\le \omega$. Furthermore, we denote $||\partial F(x)||_2 = \sup_{g\in\partial F(x)} ||g||_2$.

\noindent \textbf{Outer Problem minimization -- } Solving the \ref{eq:outer_problem} relies on the gradient descent step
\begin{equation}\label{eq:gradient_update}
    \theta^{[\ell+1]} = \theta^{[\ell]} - \tau g(\theta^{[\ell]}),
\end{equation}
where 
\begin{equation}
\label{eq:gradient_descent}
   (\forall \theta\in \Theta) \qquad    g(\theta) \in \partial_\theta (\mathcal{L}\circ\widehat{x}_{\gdot})(\theta) = \partial_x \mathcal{L}(\widehat{x}_{\theta})^\top\partial_\theta \widehat{x}_{\theta}.
\end{equation}
Since we assume that $\mathcal{L}$ does not depend directly on $\theta$,  the term $\partial_x\mathcal{L}$ is simple to compute in several cases. For instance, in the case when the $\mathcal{L}\colon x\mapsto \|x-\overline{x}\|^2/2$, it is differentiable and
$\partial_x\mathcal{L}\colon x\mapsto x-\overline{x}$.
On the other hand, the evaluation of $\partial_\theta \widehat{x}_\theta$ is more challenging, since it involves the solution of a parametric optimization problem.
Below we discuss how to approximate this subgradient. \\


\noindent \textbf{Deep Equilibrium framework} -- Deep Equilibrium \cite{bai_deep_2019,gilton_deep_2021} is a framework which aims at obtaining an expression of  $\partial_\theta \widehat{x}_\theta$ in the specific case where $ \widehat{x}_\theta$ is the fixed fixed point of an operator $\Phi_K(\cdot,\theta)$, i.e., 
\begin{equation}\label{eq:fixed_point_phi_theta}
    \widehat{x}_\theta = \Phi_K(\widehat{x}_\theta, \theta).
\end{equation}
This formulation is of particular interest for inner problems of the form \eqref{eq:min_F_plus_G} solved by fixed point iterations algorithms as described in Section~\ref{sec:basics_of_optim}.

Taking derivatives w.r.t. $\theta$ of \eqref{eq:fixed_point_phi_theta} yields
\begin{align}
    \partial_\theta \widehat{x}_\theta = [\partial_x \Phi_K(\gdot, \theta)(\widehat{x}_\theta)]\partial_\theta \widehat{x}_\theta + \partial_\theta \Phi_K (\widehat{x}_\theta,\gdot)(\theta).
\end{align}
Rearranging the terms we obtain
\begin{equation}
    [\Id-\partial_x \Phi_K(\gdot, \theta)(\widehat{x}_\theta)]\partial_\theta\widehat{x}_\theta = \partial_\theta \Phi_K (\widehat{x}_\theta,\gdot)(\theta)
\end{equation}
and provided the inversion can be performed we have
\begin{equation}
    \partial_\theta \widehat{x}_\theta = [\Id-\partial_x \Phi_K(\gdot, \theta)(\widehat{x}_\theta)]^{-1}\partial_\theta \Phi_K(\widehat{x}_\theta,\cdot)(\theta).
\end{equation}
So if we  minimize a function $\mathcal{L}(\widehat{x}_\theta)$, then the hypergradient defined in \eqref{eq:gradient_descent} is given by
\begin{align}
    g(\theta) &= \partial_x \mathcal{L}(\widehat{x}_\theta)^\top\partial_\theta \widehat{x}_\theta \label{eq:step_deq}\\ &= \partial_x \mathcal{L}(\widehat{x}_\theta)^\top[\Id-\partial_x \Phi_K(\cdot, \theta)(\widehat{x}_\theta)]^{-1}\partial_\theta \Phi_K(\widehat{x}_\theta,\cdot)(\theta).\label{eq:step_deq2}
\end{align}


\noindent \textbf{Computation of a DEQ gradient step} -- 
From the above discussion, we conclude that the appropriate hypergradient for learning the hyperparameters involved in an algorithm defined by the fixed-point equation $\widehat{x}_\theta = \Phi_K(\widehat{x}_\theta,\theta)$ is given by equation \eqref{eq:step_deq2}. In practice, this operation requires three successive steps:
\begin{enumerate}
    \item Finding $\widehat{x}_\theta$ such that $\widehat{x}_\theta = \Phi_K(\widehat{x}_\theta, \theta)$.
    \item Computing the Jacobian $\partial_x \Phi_K(\cdot, \theta)(\widehat{x}_\theta)$.
    \item Inverting the matrix 
    \begin{equation}
    \label{eq:jtheta}
        \J_\theta = [\Id-\partial_x \Phi_K(\cdot, \theta)(\widehat{x}_\theta)].
    \end{equation}
\end{enumerate}
Regarding \textbf{1.}, any fixed-point solver can be considered for $\Phi_K(\cdot, \theta)$. {Computing the Jacobian} in \textbf{2.} can be performed by \textit{automatic differentation}, yet it requires that the algorithm is not too deep.  It is important to note that if $\mathcal{H}$ is a real Hilbert space of dimension $\d$, then the Jacobian $\partial_x \Phi_K(\cdot, \theta)(\widehat{x}_\theta)$ is a matrix of size $\d^2$ which can be challenging for many applications such as encountered in image reconstruction and may restrict its computation for high-dimensional data. Finally, inverting $\J_\theta$ in \textbf{3.}  is a difficult problem due to dimensionality.  The most frequently used procedure is by approximating the Neumann inversion formula 
\begin{equation}\label{eq:neumann_j_theta}
    \mathrm{I}_P= \sum_{p=0}^P [\partial_x \Phi_K(\cdot, \theta)(\widehat{x}_\theta)]^p 
\end{equation}
which benefits from simpler operations (i.e. matrix-matrix multiplications) and is known to converge as $\lim_{P\to\infty}\mathrm{I}_P= \J_\theta^{-1}$.\\

\noindent\textbf{Truncated iterations} -- The standard theory of restart aims at estimating $\widehat{x}_\theta$, which in practice is approximated by the output after $K$ iterations of an elementary scheme  $\widetilde{x}_{\theta,K} = \Phi_K(x_0,\theta)$. The gradient step to solve the \ref{eq:outer_problem} thus reads
\begin{align}
\label{eq:gradient_descent_K}
   (\forall \theta\in \Theta) \qquad    g^K(\theta) &= \partial_\theta (\mathcal{L}\circ\Phi_K(x_0,\gdot))(\theta) \nonumber \\
   &= \partial_x \mathcal{L}(\Phi_K(x_0,\theta))^\top\partial_\theta \Phi_K(x_0,\gdot)(\theta).
\end{align}
where the $\partial_x$-term has often a closed form expression while the $\partial_\theta$-term is computed by automatic differentiation. The larger $K$ is, the closer from $\widehat{x}_\theta$ will be the solution but $K$ can not be selected too large due to computational budget issues.  

In this work, we explore a hybrid strategy in between DEQ and truncated iterates allowing to perform bilevel estimations even when truncated iterates are involved in $\Phi_K$. 

\section{Proposed Restart Jacobian Free algorithm}\label{sec:section3}

The algorithm we propose has been developed on the following observations: (i) The bilevel theoretical framework is always approximate, as $\widehat{x}_\theta$ is approached by $\widetilde{x}_{\theta, K}$ through a finite number of $K$-iterations (except for closed-form expression, which is out of the scope of this work), (ii) The theoretical framework such as DEQ relies on the knowledge of $\widehat{x}_\theta$, (iii) The cost of $\mathrm{J}_\theta^{-1}$ defined in \eqref{eq:jtheta} is prohibitive, (iv) Numerical backpropagation and automatic differentiation methods can only be implemented for $K$ "not too large" (i.e. $K\lesssim 10^2$).


In this work, we aim to benefit from properties of contractive algorithms to develop a bilevel procedure that does not require Jacobian inversion. We will theoretically assess that $\mathrm{J}_\theta^{-1} \approx \mathrm{Id}$ through contraction properties of the $\Phi_K(\cdot,\theta)$ operator with $K$ not too large - thus obtaining a feasible Jacobian free back-propagation algorithm.

\subsection{Algorithmic scheme}

Our algorithmic sheme relies on applying $T$ times a contractive operator $\Phi_{K}(\cdot,\theta)$, thereby defining a $T$-times restarted algorithm  as follows:
\begin{equation}\label{eq:restart}
   (\forall x_0\in \mathbb{R}^\d)\qquad   \Phi_{K}^T(x_0,\theta) := \underbrace{\Phi_{K}(\cdot,\theta) \circ \dots \circ \Phi_{K}}_{T \text{ times}} (x_0,\theta).
\end{equation}
Benefiting from this restart architecture, we can create an algorithmic procedure for estimating hyperparameters, where the partial derivative with respect to $\theta$ is computed only for the final restarted step of  $\Phi_{K}(\cdot,\theta)$. This procedure is described in Algorithm~\ref{alg:restarted_algorithm} and named \acronym.

\begin{algorithm}
\caption{\acronym: \textbf{Re}started \textbf{T}runcated \textbf{un}roll\textbf{e}d}\label{alg:restarted_algorithm}
\begin{algorithmic}
\Require Set $K,T > 0$ and initialize $\theta^{[0]}$.\\
$\begin{array}{l}
    \text{For } \ell = 0,  \ldots, L \\
    \left\lfloor 
    \begin{array}{l}  
    \mbox{Initialize $x_0^{[\ell]}$}.\\
     \text{For } t=1,\dots,T-1 \\
    \left\lfloor 
    \begin{array}{l}  
     x_{Kt}^{[\ell]} = \Phi_{ K}(x^{[\ell]}_{K(t-1)}, \theta^{[\ell]})
	\end{array}
    \right.\\[.4PC]
          x_{KT}^{[\ell]} = \Phi_{ K}(x_{K(T-1)}^{[\ell]}, \theta^{[\ell]})\\
    \theta^{[\ell+1]} = \theta^{[\ell]} - \eta (\partial_x \mathcal{L}(x_{KT}^{[\ell]}) )^\top \partial_\theta \Phi_{K}(x_{K(T-1)}^{[\ell]},\cdot) (\theta^{[\ell]})
	\end{array}
    \right. \\
    \text{Output: } \;\; \widetilde{x}_{\theta^{[L]},K}^T = x_{KT}^{[L]}
\end{array}$
\end{algorithmic}
\end{algorithm}%
In the proposed procedure the gradient step takes the form:
\begin{align}\label{eq:gradient_descentr}
    (\forall \theta \in \Theta) \qquad g^{\mathrm{R}}(\theta)= (\partial_x \mathcal{L}(x_{KT}) )^\top \partial_\theta \Phi_{K}(x_{K(T-1)},\cdot) (\theta)
\end{align}
where $K$ is chosen small enough to be able to perform the back-propagation procedure evaluating $\partial_\theta \Phi_{K}(x_{K(T-1)},\cdot) (\theta)$, with $x_{K t} = \Phi_{K}^t(x_0,\theta) $.

\subsection{Theoretical analysis}  

The goal of this section is twofold: establish the convergence of the sequence $(x_k)_{k\in \mathbb{N}}$ to $\widehat{x}_{\theta^{[\ell]}}$ the solution of  \eqref{eq:inner_problem} for every $\theta^{[\ell]}$, and second, to quantify the approximation error between the evaluation of the true gradient $g(\theta)$, as defined in \eqref{eq:gradient_descent},  the gradient evaluated at the optimum $\widehat{x}_\theta$, and the approximate gradient $g^R(\theta)$ considered in \acronym.
\subsubsection{Convergence of the sequence for the Inner Problem} \label{sec:convergence_inner_pb}
We recall that, for a fixed $\theta\in\Theta \subset \mathbb{R}^d$, the Inner Problem is defined by 
\begin{equation}
\label{e:inner2}
\widehat{x}_{\theta}\in\arg\min_{x\in\mathbb{R}^{\d}}f_{\theta}(x)+g_{\theta}(x).
\end{equation}

\begin{assumptions}\label{assumptions1}
For every $\theta\in\Theta$, the following hold:
    \begin{enumerate}
        \item $\Phi_K(\cdot,\theta)$ is $\delta_K(\theta)$-Lipschitz with $\delta_K(\theta)< 1$. 
        \item $\{\widehat{x}_{\theta}\}={\rm fix}\, \Phi_K(\cdot,\theta).$
    \end{enumerate}
\end{assumptions}
This assumption hold, for instance, if there exists an algorithm with elementary steps defined by
Lipschitz operators $\{\varphi_k(\cdot,\theta)\}_{k\in\NN}$ satisfying
\begin{equation}
\label{e:fixeduni}
    \{\widehat{x}_{\theta}\}=\bigcap_{k\in\NN}{\rm fix}\,\varphi_k(\cdot,\theta),
\end{equation}
and the unrolled operator over $K$ steps is
$\Phi_K(\cdot,\theta)=\varphi_K(\cdot,\theta)\circ\cdots\circ\varphi_1(\cdot,\theta)$.
Indeed, if, for every step $k\in\NN$, the operator $\varphi_k(\cdot,\theta)$ defining the single step algorithm is $\omega_k(\theta)$-Lipschitz for some $\omega_k(\theta)>0$ and $\delta_K(\theta)=\prod_{k=1}^K\omega_k(\theta)<1$, then
we have, for every $x$ and $y$ in $\mathbb{R}^d$,
$$\|\Phi_K(x,\theta)-\Phi_K(y,\theta)\|\le\omega_K\|\Phi_{K-1}(x,\theta)-\Phi_{K-1}(y,\theta)\|\le \delta_K\|x-y\|.$$
The second assumption is valid since \eqref{e:fixeduni} implies $\widehat{x}_{\theta}=\varphi_1(\widehat{x}_{\theta},\theta)$ and, hence, $\widehat{x}_{\theta}=\varphi_2(\widehat{x}_{\theta},\theta)=\varphi_2(\varphi_1(\widehat{x}_{\theta},\theta),\theta)$.
By recurrence, we obtain $\widehat{x}_{\theta}=\Phi_K(\widehat{x}_{\theta},\theta)$, which is the unique fixed point 
since $\Phi_K(\cdot,\theta)$ is $\delta_K(\theta)$-Lipschitz with $\delta_K(\theta)<1$.
In particular, these assumptions are satisfied by several first-order algorithms under strong convexity of $f_{\theta}$ and/or $g_{\theta}$. The following convergence result follows from \cite[Theorem~1.50]{Livre1}.
\begin{theorem}\label{th:banach_picard}
Suppose that $\Phi_K$ satisfies Assumption~\ref{assumptions1}. Then, for every $\theta\in\Theta$ and $x_0\in\mathbb{R}^n$,
$$(\forall t\in\NN)\quad \|x_{Kt}-\widehat{x}_{\theta}\|\le \delta_K(\theta)^t\|x_0-\widehat{x}_{\theta}\|,$$
where, for every $t\in\NN$, $x_{Kt}=\Phi_K(x_{K(t-1)},\theta)$.
Hence, $x_{Kt}\to \widehat{x}_{\theta}$ as $t\to+\infty$.
\end{theorem}

In the next section, we justify why the Jacobian Free Propagation assumption can be applied in such a context (i.e., $\mathrm{J}_\theta^{-1} = \mathrm{Id}$) and we evaluate the approximation error between the exact bilevel procedure and the proposed restarted Jacobian free algorithm \acronym.

We have thus shown that in some settings \eqref{eq:min_F_plus_G} can be solved by taking a limit in number of restarts $T$ instead of a limit in the number of elementary steps $\varphi_k$. This is relevant for practical reasons since it means that one can fix the architecture of an algorithm $\Phi_{K}$ and that it can still exactly solve \eqref{eq:min_F_plus_G} provided sufficiently many restarts (power-iterates) are performed.\\

\subsubsection{Error on gradient step} 
To facilitate the error quantification between $g(\theta)$, the true gradient defined in~\eqref{eq:gradient_descent}, and $g^R(\theta)$, the approximate gradient obtained from Algorithm~\ref{alg:restarted_algorithm}, we introduce a third type of gradient step used in JFB \cite{fung_jfb_2022}, whose expression is
\begin{equation}\label{eq:gradient_descent_jf}
    g^{\JF}(\theta) = \partial_x \mathcal{L}(\widehat{x}_\theta)^\top \partial_\theta \Phi_K(\widehat{x}_\theta,\cdot)(\theta).
\end{equation}
We start by providing three different equivalent interpretations of \eqref{eq:gradient_descent_jf}:
\begin{enumerate}
    \item It stands for  \eqref{eq:step_deq} when $\J_\theta^{-1} = \Id$;
    \item The Neumann inversion \eqref{eq:neumann_j_theta} is truncated at $P=0$, leading to $\J_\theta^{-1} = \Id$;
    \item In  \eqref{eq:step_deq2}, $\partial_x\Phi_K(\cdot, \theta)(\widehat{x}_\theta) = 0$;
\end{enumerate}
The first two statements are at the core of the original JFB paper \cite{fung_jfb_2022} and reach the conclusion that this allows in many scenarios to obtain a training performance usually higher and faster than DEQ with Neumann inversion. In what follows we will justify that the third statement can be almost satisfied in our setting. 


The following result allows us to quantify the error between the exact gradient $g$ defined in \eqref{eq:step_deq2} and the JFB gradient  $g^{\mathrm{JF}}$ defined in \eqref{eq:gradient_descent_jf}, both involving the exact knowledge of $\widehat{x}_\theta$.

\begin{lemma}\label{th:prop_jf} If $\mathcal{L}$ is differentiable and Assumption~\ref{assumptions1} holds, then
    \begin{equation}
        ||g(\theta) - g^{\mathrm{JF}}(\theta)||_2\leq \frac{\delta_K(\theta)}{1-\delta_K(\theta)}||\partial_x \mathcal{L}(\widehat{x}_{\theta})||_2 ||\partial_\theta \Phi_{K} (\widehat{x}_{\theta},\cdot)(\theta)||_2.
    \end{equation}
\end{lemma}
\begin{proof}
The proof provided below relies on standard computations except for the last inequality that comes in applying Lemma~\ref{lem:lipschitz_matrice} of the Appendix with $H = \partial_x\Phi_K(\cdot, \theta)(\widehat{x}_\theta)$ to obtain $||\Id - \J_\theta^{-1}||_2 \leq \frac{\delta_K(\theta)}{1-\delta_K(\theta)}$. From this we get 
\begin{align}
    ||g(\theta) - g^{\mathrm{JF}}(\theta)||_2 &=  ||\partial_x \mathcal{L}(\widehat{x}_{\theta}) (I-\J_\theta^{-1})\partial_\theta \Phi_{K} (\widehat{x}_{\theta},\cdot)(\theta)||_2\\
     &\leq  ||\partial_x \mathcal{L}(\widehat{x}_{\theta})||_2|| (I-\J_\theta^{-1})||_{2}||\partial_\theta \Phi_{K} (\widehat{x}_{\theta},\cdot)(\theta)||_2\\
     &\leq \frac{\delta_K(\theta)}{1-\delta_K(\theta)}||\partial_x \mathcal{L}(\widehat{x}_{\theta})||_2 ||\partial_\theta \Phi_{K} (\widehat{x}_{\theta},\cdot)(\theta)||_2.
\end{align}
\end{proof}%
\noindent If the iterative scheme $\Phi_K$ has Lipschitz constants satisfying $\delta_K \to 0$ as $K \to \infty$ - as with the Picard iterations in Section \ref{sec:convergence_inner_pb} - then Lemma~\ref{th:prop_jf} directly implies that the error between $g$ and $g^{\mathrm{JF}}$ becomes arbitrarily small for large $K$. 

 This suggests that gradient steps at equilibrium of some architectures of \textit{truncated unrolled} schemes can be well approximated by $g^\mathrm{JF}$ provided they are sufficiently deep, eliminating the need to compute and invert $\J_\theta$ in the DEQ setting.

\subsection{Approximation error}  
We complete the result of Lemma~\ref{th:prop_jf} by considering the situation where $\widehat{x}_\theta$ is not exactly known. By Assumption~\ref{assumptions1}, we can use the restart scheme \eqref{eq:restart} in order to construct an approximating sequence of $\widehat{x}_\theta$. The following result gives an error bound on the DEQ gradient step at an exact fixed point, and the JFB gradient step at a point $x_{KT}$ obtained through the \acronym\ scheme. We need the following assumption.
\begin{assumption}\label{assumptions2}
 For every $\theta\in\Theta$, $\partial_\theta \Phi_K(\cdot,\theta)$ is
 locally Lipschitz continuous at $\widehat{x}_{\theta}$,
 i.e., for some $\delta>0$, there exists $L_{\theta}>0$
 such that, for every $x\in\mathbb{R}^n$ satisfying $\|x-\widehat{x}_{\theta}\|_2\leq \delta$, we have
 $$
 \|\partial_{\theta}\Phi_K(x,\theta)-\partial_{\theta}\Phi_K(\widehat{x}_{\theta},\theta)\|_2\le L_{\theta}\|x-\widehat{x}_{\theta}\|_2.$$
\end{assumption}
This asumption hold, for instance, if each step $k\in\NN$
of the algorithm $\varphi_k$ is of class $\mathcal{C}^2$. Indeed, since $\Phi_K$ is the composition of $\mathcal{C}^2$ functions, it is $\mathcal{C}^2$, which implies that $\partial_\theta \Phi_K(\cdot,\theta)$ is $\mathcal{C}^1$ and, therefore, $\partial_x\partial_\theta \Phi_K(\cdot,\theta)$ is continuous and bounded in compact sets. Therefore, the gradient descent algorithm of a $\mathcal{C}^3$
function satisfies this assumption.

\begin{theorem}\label{th:main_th} If $\mathcal{L}$ is twice differentiable and Assumptions~\ref{assumptions1} and \ref{assumptions2} are satisfied, then
    \begin{align}
        \|g(\theta) &- g^{\mathrm{R}}(\theta)\|_2 \leq  \frac{\delta_K}{1-\delta_K} \|\partial_x \mathcal{L}(\widehat{x}_\theta)\|_2\|\partial_\theta \Phi_{K}(\widehat{x}_{\theta},\cdot)\|_{2} \nonumber\\
        &+ \delta_K^TL_{\theta}\|\widehat{x}_{\theta} - x_0\|_2 \|\partial_x \mathcal{L}(\widehat{x}_\theta)\|_2 \nonumber \\
        &+ \delta_K^T\|\widehat{x}_\theta - x_0\|_2
        \|\partial_{xx}\mathcal{L}(\widehat{x}_\theta)\|_2 \| \partial_\theta \Phi_{K}(x_{K(T-1)},\cdot)(\theta)\|_{2}\nonumber\\
        &+ \|\partial_\theta \Phi_K(x_{K(T-1)},\cdot)(\theta)\|_2 o\big( \delta_K^T\|\widehat{x}_\theta - x_0\|_2\big).\label{eq:bound_th2}
    \end{align}
\end{theorem}
\noindent By inspecting the bound in \eqref{eq:bound_th2} we have that the first term is the same as the one obtained in Lemma~1 and every other term has a factor $\delta_K^T$. Thus for a contractive algorithm, for which by definition $\delta_K<1$, then as the number of restarts increases every term but the first tend linearly to 0. This result motivates the introduction of Algorithm~\ref{alg:restarted_algorithm} with $T$ and $K$ as large as possible so as to have a minimal error between $g$ and $g^{\mathrm{R}}$. 
\begin{proof}
    By the triangle inequality
    \begin{align}
        \|g(\theta) - g^{\mathrm{R}}(\theta)\| \leq     \|g(\theta) - g^{\mathrm{JF}}(\theta)\| +     \|g^{\mathrm{JF}}(\theta) - g^{\mathrm{R}}(\theta)\|.
    \end{align}
    The first term of the r.h.s. of the inequality has already been computed in Lemma~\ref{th:prop_jf}. We compute the last term as follows
    by successively using: the identity $\partial_x \mathcal{L}(x_{KT})= \partial_x \mathcal{L}(x_{KT}) - \partial_x\mathcal{L}(\widehat{x}_{\theta}) + \partial_x\mathcal{L}(\widehat{x}_{\theta})$; the triangle inequality; and several times Assumptions~\ref{assumptions1} and~\ref{assumptions2}. From Theorem~\ref{th:banach_picard} and Assumption~\ref{assumptions2}, it follows that for any $x_0$ there exists $T$ large enough such that the bound in Assumption~\ref{assumptions2} with $x=x_{KT}$ is satisfied for some $L_\theta > 0$. This yields 
    \begin{align}
        \|g^{\mathrm{JF}}(\theta) &- g^{\mathrm{R}}(\theta)\|_2 \\
        =&  \|\partial_x\mathcal{L}(\widehat{x}_{\theta}) \partial_\theta \Phi_{K}(\widehat{x}_{\theta},\cdot)(\theta) - \partial_x \mathcal{L}(x_{KT}) \partial_\theta \Phi_{K}(x_{K(T-1)},\cdot)(\theta)\|_2 \\
        =&  \|\partial_x\mathcal{L}(\widehat{x}_{\theta})\big( \partial_\theta \Phi_{K}(\widehat{x}_{\theta},\cdot)(\theta) -  \partial_\theta \Phi_{K}(x_{K(T-1)},\cdot)(\theta)\big)\nonumber\\
        &- (\partial_x \mathcal{L}(x_{KT}) - \partial_x\mathcal{L}(\widehat{x}_{\theta})) \partial_\theta \Phi_{K}(x_{K(T-1)},\cdot)(\theta) \|_2  \\
        \leq&  \|\partial_x\mathcal{L}(\widehat{x}_{\theta})( \partial_\theta \Phi_{K}(\widehat{x}_{\theta},\cdot)(\theta) -  \partial_\theta \Phi_{K}(x_{K(T-1)},\cdot)(\theta))\|_2 \nonumber\\
        &+ \| (\partial_x \mathcal{L}(x_{KT}) - \partial_x\mathcal{L}(\widehat{x}_{\theta})) \partial_\theta \Phi_{K}(x_{K(T-1)},\cdot)(\theta) \|_2  \\
        \leq &  \|\partial_x\mathcal{L}(\widehat{x}_{\theta})\|_2\| \partial_\theta \Phi_{K}(\widehat{x}_{\theta},\cdot)(\theta) - \partial_\theta \Phi_{K}(x_{K(T-1)},\cdot)(\theta)\|_2\nonumber\\
        &+ \| \partial_x \mathcal{L}(x_{KT}) - \partial_x\mathcal{L}(\widehat{x}_{\theta})\|_2\| \partial_\theta \Phi_{K}(x_{K(T-1)},\cdot)(\theta) \|_2  \\
        \leq & \|\partial_x\mathcal{L}(\widehat{x}_{\theta})\|_2 L_{\theta}\| \widehat{x}_{\theta} - x_{KT}\|_2\nonumber\\
        &+  \| x_{KT} - \widehat{x}_\theta\|_2\|\partial_{xx}\mathcal{L}(\widehat{x})\|_2  \| \partial_\theta \Phi_{K}(x_{K(T-1)},\cdot)(\theta) \|_2 \nonumber\\
        &+o(\| x_{KT} - \widehat{x}_\theta\|_2)\| \partial_\theta \Phi_{K}(x_{K(T-1)},\cdot)(\theta) \|_2  \\
        \leq & \delta_K^TL_{\theta} \|\partial_x\mathcal{L}(\widehat{x}_{\theta})\|_2 \| \widehat{x}_{\theta} - x_0\|_2\nonumber\\
        &+  \delta_K^T\| x_0 - \widehat{x}_\theta\|_2\|\partial_{xx}\mathcal{L}(\widehat{x})\|_2\| \partial_\theta \Phi_{K}(x_{K(T-1)},\cdot)(\theta) \|_2 \nonumber \\
        &+o\big(\delta_K^T\|\widehat{x}_\theta - x_0\|_2  \big) \|\partial_\theta \Phi_K(x_{K(T-1)},\cdot)(\theta)\|_2.
    \end{align}    
\end{proof}
As an application we specify the bound of Theorem~\ref{th:main_th} when the outer loss is the mean-square error.
\begin{corollary}
        Let $\overline{x} \in \mathbb{R}^\d$ and set $\mathcal{L}(x) = \frac{1}{2}\|x-\overline{x}\|^2$, then
    \begin{align}
        &\|g(\theta) - g^{\mathrm{R}}(\theta)\|_2 \leq \frac{\delta_K}{1-\delta_K} \|\widehat{x}_\theta - \overline{x}\|_2\|\partial_\theta \Phi_{K}(\widehat{x}_{\theta},\cdot)\|_{2} \nonumber \\
        &+ \delta_K^T\|x_0 - \widehat{x}_{\theta}\|_2 \left(  \| \partial_\theta \Phi_{K}(x_{K(T-1)},\cdot)(\theta)\|_{2}  + L_\theta\|\widehat{x}_\theta - \overline{x}\|_2 \right).
    \end{align}
\end{corollary}
\begin{proof}
If $\mathcal{L}(x) = \frac{1}{2}\|x-\overline{x}\|_2^2$, then $\partial_x \mathcal{L}(x) = x-\overline{x}$ and $\partial_{xx} \mathcal{L}(x) = \Id$. It may also be noted that $\|\partial_x \mathcal{L}(x)\|_2 = \|x-\overline{x}\|_2 = \sqrt{2\mathcal{L}(x)}$.
\end{proof}




\section{Numerical experiments}\label{sec:section4}
This section evaluates the performance of Algorithm~\ref{alg:restarted_algorithm} for inverse problems of the form $y=A\overline{x} + \epsilon$ with known forward operator $A$ and  Gaussian noise $\epsilon$. A first set of experiments is conducted under Assumption~\ref{assumptions1} used to match part of the theoretical analysis. A second set of experiments targets a more complex image restoration scenario, aiming not only to assess state-of-the-art performance in terms of reconstruction but also to evaluate the robustness of the proposed algorithm in a more realistic context.
\subsection{Wavelet denoising}
\noindent \textbf{Forward model} -- We consider a denoising problem (e.g., $A=\Id$) with channel dependent noise, which is defined, for an image $x = (x_{R},x_{G},x_{B})$, as $\epsilon = (\epsilon_R, \epsilon_G, \epsilon_B)$ with independent realisations of white Gaussian noise with respective standard deviation $\sigma_R, \sigma_G$, and $\sigma_B$. In our simulations, the noise standard deviations are assumed unknown.

\noindent \textbf{Energy} -- The estimator $\widehat{x}_\theta$ of $\overline{x}$ is obtained by a variational formulation involving a parametric prior relying on a wavelet transform, denoted $D \in \mathbb{R}^{\d\times \d} $. This leads to
\begin{equation}
\label{eq:deoisingweights}
    \widehat{x}_\theta = \argmin_x \frac{1}{2}\|x - y\|_2^2 + \Vert D x \Vert_{1,2,\theta} 
\end{equation}
where $\Vert \cdot \Vert_{1,2,\theta}$ is a parametric prior imposing regularity on wavelet coefficients. In this work we consider two variants that jointly penalize across bands (horizontal, vertical, diagonal) and/or channels (RGB):
 \begin{enumerate}
     \item $\ell_{2,1}$-norm with weighted bands and channels: 
     \begin{equation}
     \label{eq:priorBC}
        (\forall w\in \mathrm{R}^d) \quad \|w\|_{1,2,\theta_{\textrm{B,C}}} := \sum_j \lambda_j \sum_k\sqrt{\sum_{b,c} \Lambda_{b,c} w_{j,k,b,c}^2}; 
     \end{equation}
     \item $\ell_{2,1}$-norm with weighted bands: 
     \begin{equation}
          \label{eq:priorB}
          (\forall w\in \mathrm{R}^d) \quad \|w\|_{1,2,\theta_{\textrm{B}}} := \sum_j \lambda_j \sum_{k,c}\sqrt{\sum_{b} \Lambda_{b} w_{j,k,b,c}^2}.
     \end{equation}
 \end{enumerate}
In both settings, the parameters $\lambda_\cdot$ and $\Lambda_\cdot$ are positive, furthermore, the indices $j,k,b,c$ respectively denote the scale, position, orientation and channel of a wavelet coefficient $w_{j,k,b,c}$. These two norms can be compactly expressed as
 $\Vert D w \Vert_{1,2,\theta} = \Vert \theta D w \Vert_{1,2}$ where $\theta$ encapsulates both $\Lambda$ and $\lambda$. More precisely, $\theta$ is a linear diagonal operator which acts by multiplying each wavelet coefficient $w_{j,k,b,c}$ by $\lambda_j\sqrt{\Lambda_{b,c}}$.\\

The minimization problem~\eqref{eq:deoisingweights}, can be expressed as the proximity operator of the weighted norm, and, due to orthonormality of $D$, can be equivalently  reformulated as:
 \begin{align}
    \widehat{x}_\theta  &= \prox_{\Vert \theta D \cdot \Vert_{1,2}}(y)\\
    &= D^\top\prox_{\Vert \theta \cdot \Vert_{1,2}}(Dy)\\
    &= D^\top\argmin_w \frac{1}{2}\|w - D y\|_2^2 + \Vert \theta w \Vert_{1,2}\\
    &= {D^\top\argmin_w \frac{1}{2}\|D^\top w - y\|_2^2 + \Vert \theta w \Vert_{1,2}}\\    
        &= D^\top\theta^{-1}\argmin_u \frac{1}{2}\|D^\top \theta^{-1}u - y\|_2^2 + \Vert u \Vert_{1,2}. \label{eq:minfinal_denoising}
\end{align}

\noindent \textbf{Forward-backward iterations} --  In order to estimate $\widehat{x}_\theta$, we focus on the minimization problem formulated in \eqref{eq:minfinal_denoising}, leading to $f_\theta = \frac{1}{2}||D^\top \theta^{-1}\cdot - y||_2^2$ and $g_\theta =\Vert \cdot \Vert_{1,2}$  allowing us to apply elementary forward-backward steps on the wavelet coefficients expressed as:
 \begin{equation}
    u_{k+1} =  \varphi(u_k,\theta) := \prox_{\tau||\cdot||_{1,2}} \left( u_k - \tau \theta^{-1}D (D^\top \theta^{-1} u_k - y) \right) 
     \end{equation}
To insure convergence $\tau < 2/L$ with $L = 1/\min(\theta)^2$, the optimal choice is given by  $\tau^* =  2/(1/\max(\theta)^2 + 1/\min(\theta)^2)$\cite{Sigpro1}. 



\noindent\textbf{Truncated unrolled scheme and hyperparameters to learn -- } This forward-backward iteration leads to the following $K$-times truncated unrolled scheme
 \begin{equation}\label{eq:algo_wav_denoising}
     \Phi_K(x,\theta) = \varphi^K (x,\theta)
 \end{equation}

\noindent Based on \eqref{eq:algo_wav_denoising},  we define the restarted scheme as $\Phi_{K}^T$,  which is guaranteed to converge to $\widehat{x}_\theta$ defined in \eqref{eq:minfinal_denoising} for any initialization $x_0$ as $K$ or $T$ get arbitrarily large (cf. Section~\ref{sec:convergence_inner_pb}). 

We denote $\widetilde{x}_{ \theta,K}^{T} = D^\top  \theta^{-1}   \Phi_{K}^T(x_0, \theta)$ with $x_0 = D y$ the estimate obtained after $\Phi_{K}$ scheme and $T$ restarts.  The hyperparameter to learn are either associated with weighted band and channels \eqref{eq:priorBC} or only weighted bands \eqref{eq:priorB}. 
In our experiments, we consider the  Daubechies-4 wavelet transform \cite{daubechies_ten_1992,mallat_wavelet_2008} and compute 4 levels of detail coefficients. 

\noindent \noindent \textbf{Validity of the assumptions -- } Since $f_\theta$ is strongly convex and $g_\theta$ is convex for all choices of  $\theta\in\Theta$, then $E_\theta$ is strongly convex. By defining a stepsize as $\tau^{[\ell]} = 1.95/L^{[\ell]}$ with $L=1/\min(\Lambda^{[\ell]})$ then $\varphi_k$ is ($\delta < 1$)-Lipschitz continuous.

\noindent \textbf{Outer problem minimization -- }
We consider a training data-set $\mathcal{D} = \{(\overline{x}^{(i)}, y^{(i)})\}_{i\in \mathbb{I}}$  consisting of $\vert \mathbb{I}\vert =600$ pairs of clean and degraded images of size $\d = 256\times 256$ extracted from the Div2K dataset \cite{Agustsson_2017_CVPR_Workshops} with pixel intensities in $[0,1]$. The learning stage aims to minimize the loss $\mathcal{L}$, which is the mean-squared error between $\widetilde{x}_{\theta,K}^{T,(i)}$ and $\overline{x}^{(i)}$ evaluated on a batch of 4 images. The algorithm is run over $E = 4$ epochs, leading to a total of $\frac{600}{4}\times E$ gradient evaluation. We minimize $\mathcal{L}$ using the optimizer ADAM \cite{kingma_adam_2017} with a step-size set to $5.10^{-2}$. 
We extracted 4 images from the testing dataset that were not used for the training stage, in order to evaluate the reconstruction performance on them. 
 

\noindent \textbf{Ensuring positivity constraints} --  All parameters such as $\lambda$ or the weights $\Lambda$ must be positive to ensure the well-posedness of the problem. To ensure positivity, we learn $\widetilde{\lambda} \in \mathbb{R}$ and we always (implicitly) consider $\lambda = \exp(\widetilde{\lambda})\in\mathbb{R}_+$ whenever we have to ensure a positiveness  constraint. The same procedure is applied for the estimation of $\Lambda$.

 
\noindent \textbf{Results: Impact of $K$ and $T$ -- } In this preliminary set of experiment, we evaluate the impact of the number of layer $K$ and the number of restarts $T$.  We thus run four configurations: $(K,T) = (1,1)$, $(K,T) = (1,10)$, $(K,T) = (10,1)$, and $(K,T) = (10,10)$. The configuration $(K,T) = (10,1)$ stands for the standard unrolled neural network strategy while $(K,T) = (10,10)$ fully benefits of the proposed framework. The associated results are presented in Figure~\ref{fig:wavelet-denoising}.
\begin{figure}[htbp]
\centering
    \includegraphics[width=0.75\linewidth]{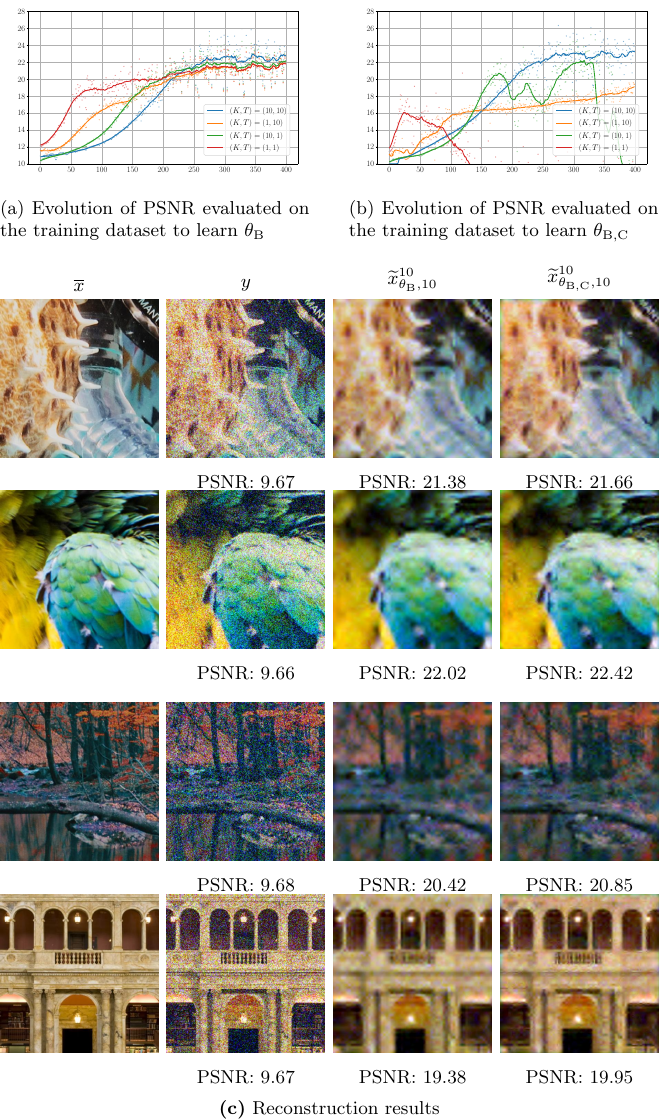}
    \caption{ Wavelet denoising for degraded data with standard deviations set to $(\sigma_{\mathrm{R}}, \sigma_{\mathrm{G}}, \sigma_{\mathrm{B}}) = (0.1,0.25,0.5)$.
    Top: PSNR across gradient descent steps $\ell$ using \acronym\ (i.e. Alg.~\ref{alg:restarted_algorithm}) for different choices of $T$ and $K$. Exact PSNR values are displayed with dots, and solid lines indicate local averages over 25 steps. 
    Bottom: From left to right: ground truth $\overline{x}$, observation $y=\overline{x} + \epsilon$, and reconstructed images $\widetilde{x}_{K=10,\theta^{[\ell = \ell_{\max}]}}^{T=10}$ under priors $\theta_{\mathrm{B}}$ and $\theta_{\mathrm{B,C}}$, including a normalized step variant. }
    \label{fig:wavelet-denoising}
\end{figure}
Figures~\ref{fig:wavelet-denoising}(a) and (b) show the evolution of PSNR for the training dataset as a function of  the number of gradient step. We observe that for both configurations (whether learning  $\theta_{\textrm{B,C}}$ or $\theta_{\textrm{B}}$), the highest PSNR is consistently and more rapidly achieved with $(K,T) = (10,10)$. Moreover, the more complex learning configuration, which stands when weights are associated to both bands and channels, is more challenging to optimize for other settings than $(K,T) = (10,10)$.



Figure~\ref{fig:wavelet-denoising} (c) also displays the restoration results obtained with the configuration $K=T=10$ on four images.  The proposed procedure allows us to provide an efficient estimation of the hyperparameters for both configurations ($\theta_{\textrm{B}}$ or $\theta_{\textrm{B,C}}$) allowing to improve the restoration performance when increasing the number of hyperparameters.

\subsection{Image restoration}
In this section, we relax the strong convexity assumption by considering data-fidelity terms involving a non-injective forward model $A$.

\noindent\textbf{Forward models and data-fidelity terms -- }  We consider two types of  forward linear models: inpainting and debluring. 

For inpainting, $A$ is a diagonal linear operator whose diagonal values $m\in\{0,1\}^\d$ are such that $m_i = 1$ when the pixel is kept and 0 otherwise. This yields to $A\colon x\mapsto  m \odot x$ where $\odot$ denotes the Hadamard product. In this setting $\frac{1}{2}||Ax - y||_2^2$ is not strongly convex.

In the deconvolution setting,  $A$ is a blurring operator defined by convolution with a kernel $h$ such that $A\colon x\mapsto x\star h$. We denote $F$ the Fourier Transform, then $\frac{1}{2}||Ax - y||_2^2 = \frac{1}{2}||(F h) \odot (F x) - F y||_2^2$ which is $\frac{m_h}{2}$-strongly convex with $m_h = \min_\xi |(F h)(\xi)|$. In general $m_h = 0$ (or $m_h\approx 0$) and the data-fidelity term is not strongly convex. We consider an anisotropic channel-dependent blur modeled by a normalized uniform point spread function with a support width of 25 pixels, which respectively averages the red channel horizontally, the green channel vertically and the blue channel diagonally. 


\noindent \textbf{Forward-Backward PnP --}
We further relax the assumption on $\Phi_K$ by considering an implicit prior. More precisely, following ideas exposed in \cite{Ulugbek_bouman_pnp_2023,hurault_proximal_2022,repetti_dual_2022}, we replace the proximal step in the forward-backward iterations by a neural network denoiser denoted $\mathrm{D}_{\Lambda,\sigma}$. In our experiments, we consider a Drunet denoiser\cite{zhang_dpir_j-tpami_2021}, where the noise level $\sigma$ corresponds to the level of regularization which should be applied to the iterates while $\Lambda$ denotes the neural networks parameters. An elementary step of FB-PnP reads
\begin{equation}\label{eq:pnp_iteration}
   (\forall x\in \mathrm{R}^\d) \quad  \varphi(x,\theta) := \mathrm{D}_{\Lambda,\sigma} (x - \tau A^\top (A x - y)). 
 \end{equation}

\noindent\textbf{Hyperparameters to learn -- }  We consider two learning configurations. 

The first one focuses on the learning of the regularization parameter $\sigma$ and the step-size $\tau$ while the denoiser has fixed (already learned) parameters leading to $\theta = (\sigma,\tau)$. More precisely, we use the weights from \cite{tachella_deepinverse_2023}, which were obtained by training over the DPIR dataset \cite{zhang_dpir_j-tpami_2021}. Despite the fact that the noise distribution is unknown, the proposed method offers a simple and efficient framework to learn the hyperparameters.  The tuning of the hyperparameters $\sigma$ and $\tau$ is crucial in order to obtain good estimates.   

In order to further improve the reconstruction results, the second learning configuration aims to learn the denoiser parameters. This allows to further refine the denoiser so that it is jointly adapted to the distribution of images in the dataset and to the forward model under consideration. In this second experiment,  we consider $\theta = \Lambda$ .

\noindent \textbf{Training --} The procedure is the same as for wavelet denoising, except that the learning rate is set to $5 \times 10^{-5}$ in the second configuration. The positivity of $\tau$ and $\sigma$ is enforced in the same manner as in the previous experiment.  In the first configuration the learned parameters are $\theta = (\sigma, \tau)$. For the second configuration  $\theta = \Lambda$ while the values of $(\sigma, \tau)$ are taken from the results of the first configuration.

\noindent\textbf{Results for inpainting -- }
In Fig.~\ref{fig:inpainting90-results} we show the results obtained for an inpainting task where 90\% of the pixels are missing. We compare three types of reconstruction: a) learning  the weight parameters $\theta_{B,C}$ of the wavelet (explicit) prior as presented in the previous section; b) learning the step-size $\tau$ and regularization parameter $\sigma$ of the iteration \eqref{eq:pnp_iteration} with implicit prior $D_{\Lambda,\sigma}$;  c) learning the denoiser $D_{\Lambda,\sigma}$ involved in the iterations \eqref{eq:pnp_iteration}.
For these 3 experiments, we evaluate the impact of the number of layers $K$ and the number of restarts $T$.  We run the same four configurations: $(K,T) = (1,1)$, $(K,T) = (1,10)$, $(K,T) = (10,1)$, and $(K,T) = (10,10)$.

In Fig.~\ref{fig:inpainting90-results} (a)-(c); for all the experiments,  we observe that the blue curve associated with the proposed \acronym\ scheme (i.e., maximal number of restarts and larger number of layers) outperforms all the other methods. Also, the two configurations (orange and green) which involve the same number of $\varphi$ steps for the evaluation of the estimate reach a similar performance, with a slight improvement for the scenario involving a restart procedure (in green). 



For the setting c), the parameters $\tau$ and $\sigma$ are initialized with the values obtained at the end of learning in the setting b). Similarly, $D_\sigma$ is initialized as the denoiser used in the setting b). We observe that learning the denoiser allows to improve the performance compared to the previous setting where the denoiser was fixed. Additionally, the procedure with $K=T=10$ allows us to achieve the best results much faster.

In Fig.~\ref{fig:inpainting90-results}(d), we display the reconstruction results obtain with each of the three configurations when $K=T=10$. We can clearly observe the benefit of learning the denoiser with \acronym\ compared to PnP and obviously compared to a standard wavelet denoising. 


\begin{figure}[htbp]
\centering
    \includegraphics[width=0.8\linewidth]{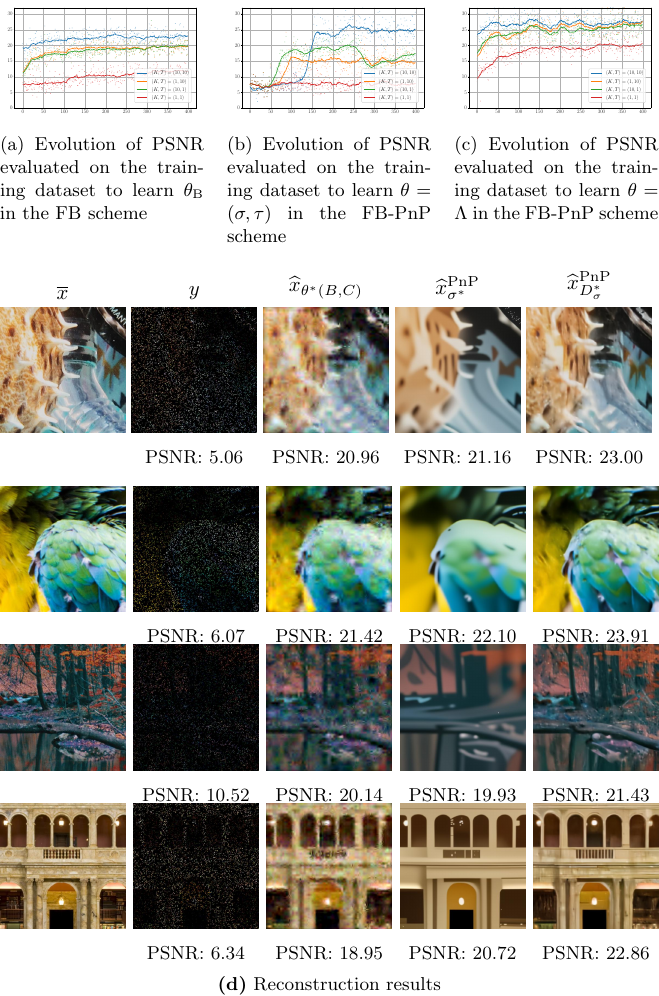}
    \caption{Results for inpainting where 90\% of the pixels are masked and noise standard deviations are set to $\sigma_{\mathrm{R}} = \sigma_{\mathrm{G}} = \sigma_{\mathrm{B}} = 0.05$. See Fig.~\ref{fig:wavelet-denoising} for description of PSNR plots. Three reconstruction results are displayed, a) FB iterations with wavelet prior $\theta(B,C)$, b) FB-PnP with learned $\sigma$ with a fixed pretrained Drunet denoiser, c) FB-PnP with a learned Drunet denoiser. d) Reconstruction estimate at the end of the learning when $K=T=10$. }
    \label{fig:inpainting90-results}
\end{figure}
\noindent\textbf{Results for deblurring --} Similar experiments are made for the deblurring problem and the results are displayed in Fig.~\ref{fig:deblurring-results}. The benefit of the restart for larger value of $K$ and $T$ is not as visible in terms of PSNR as for the inpainting problem but select $K=T=10$ still allows to improve the performance. 


Visually, we may observe in the parrot and forest images that the fixed denoiser created several artifacts which do not appear when the denoiser has been learned with the full potential of \acronym\ (i.e., $K=T=10$) in an unrolled manner.

\section{Conclusion}
In this work, we aimed to train truncated unrolled algorithm using automatic differentiation techniques with theoretical guarantees. Our analysis showed that a truncated neural network can be restarted to infer solutions of \eqref{eq:inner_problem} with fixed parameters. Furthermore, we introduced the \acronym~ procedure to optimize \eqref{eq:outer_problem} in a simple fashion, with learning steps that can be made arbitrarily close to the optimal DEQ steps. To obtain this result, Lipschitz properties of Forward-Backward iterations were leveraged to show the interplay between the depth and the number of restarts in \acronym. In particular, the depth controls the approximation error of a Jacobian-free step, whereas restarting the unrolled neural network allows one to reach the equilibrium point. This theoretical analysis is supported by numerical experiments showing that, even when all assumptions are not satisfied, the \acronym~ procedure allows to go beyond traditional learning schemes and improves performance with respect to either unrolled scheme or PnP strategies. 

\noindent\textbf{Future work.} For clarity of the presentation, the theoretical analysis was restricted to Forward-Backward iterations over a strongly convex functional. This choice allowed us to satisfy Assumption~\ref{assumptions1}, since the elementary steps are Lipschitz contractive with a global Lipschitz constant decreasing linearly with $K$. For non-strongly convex functionals, the previous statement no longer holds; yet, many optimization schemes satisfy that the (pointwise) Lipschitz constant of $K$ iterations tends to 0. Extending the \acronym~ framework to such situations seems promising, because both the equilibrium point and the Jacobian-Free Backpropagation (JFB) error can respectively be guaranteed to be reached by a restart procedure and to decrease to 0. This is also supported by our numerical results on non-strongly convex functionals.

Results on optimization schemes satisfying Assumption~\ref{assumptions2} are, however, less common, and several directions can be considered to find bounds similar to those of Theorem~\ref{th:main_th}. One idea would be to develop smooth regularizations of iterative schemes designed to satisfy Assumption~\ref{assumptions2}. Another approach would be to consider alternative notions of differentiation for non-continuous or set-valued functions—in the same spirit as the Clarke Jacobian of Lipschitz (non-smooth) functions used here—which may still provide useful bounds.

\begin{figure}[htbp]
\centering
\includegraphics[width=0.8\linewidth]{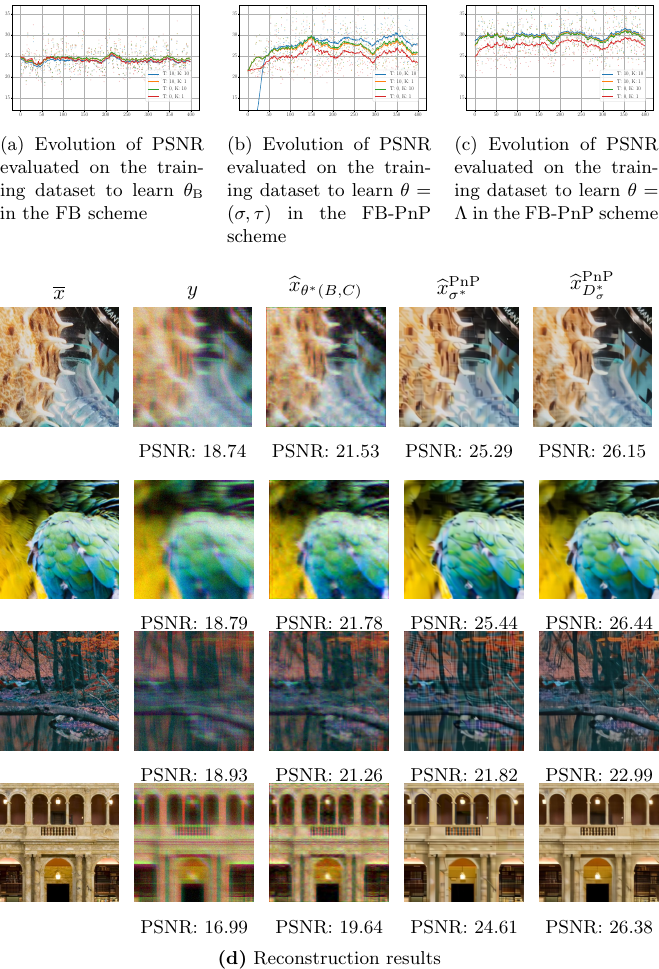}
    \caption{Results for anisotropic deblurring with $A$ a channel-wise anisotropic blur kernel of width 25 and noise standard deviations are set to $\sigma_{\mathrm{R}} = \sigma_{\mathrm{G}} = \sigma_{\mathrm{B}} = 0.05$. See Fig.~\ref{fig:wavelet-denoising} for description of PSNR plots. Three reconstruction results are displayed, a) FB iterations with wavelet prior $\theta(B,C)$, b) FB-PnP with learned $\sigma$ with a fixed pretrained Drunet denoiser, c) FB-PnP with a learned Drunet denoiser. d) Reconstruction estimate at the end of the learning when $K=T=10$.}
    \label{fig:deblurring-results}
\end{figure}

\clearpage
\section{Appendix}
\begin{lemma}[\cite{riesz-nagy}, page 152]\label{lem:lipschitz_matrice}
    Let $H$ a square matrix (or operator $\mathcal{H}\to\mathcal{H}$) with spectral norm $||H||_2 = \omega < 1$, then 
    \begin{equation}
        || I - (I-H)^{-1}||_2 \leq \frac{\omega}{1-\omega}.
    \end{equation}
\end{lemma}
\begin{proof}
    Since $||H||_2 < 1$, it follows from Neumann's expansion that
    \begin{equation}
       I - (I-H)^{-1} = I - \sum_{k\in\mathbb{N}} H^k = -\sum_{k=1}^{\infty}H^k = - H\sum_{k=0}^{\infty}H^k = -H(I-H)^{-1}.
    \end{equation}
    By the multiplicative property of the spectral norm $||H_1H_2||_2\leq ||H_1||_2||H_2||_2$ for any $H_1,H_2$ it follows that
    \begin{equation}
        ||I-(I-H)^{-1}||_2 \leq ||H||_2||(I-H)^{-1}||_2 \leq \frac{\omega}{1-\omega}
    \end{equation}
    where we used again Neumann's expansion since $\omega < 1$
    \begin{equation}
        ||(I-H)^{-1}||_2 = ||\sum_{k\in\mathbb{N}} H^k||_2 \leq \sum_{k\in\mathbb{N}} ||H^k||_2 \leq \sum_{k\in\mathbb{N}} \omega^k = \frac{1}{1-\omega}.
    \end{equation}
\end{proof}

\bibliographystyle{plain}

\end{document}